\documentclass[12pt]{amsart}

\usepackage{psfrag}
\usepackage{color}

\usepackage{graphicx,graphics}
\usepackage{fullpage,amssymb,amsfonts,amsmath,amstext,amsthm,amscd}
\usepackage[T1]{fontenc}

\begin{document}

\newtheorem{theorem}{Theorem}[section]
\newtheorem{result}[theorem]{Result}
\newtheorem{fact}[theorem]{Fact}
\newtheorem{conjecture}[theorem]{Conjecture}
\newtheorem{definition}[theorem]{Definition}
\newtheorem{lemma}[theorem]{Lemma}
\newtheorem{proposition}[theorem]{Proposition}
\newtheorem{remark}[theorem]{Remark}
\newtheorem{corollary}[theorem]{Corollary}
\newtheorem{facts}[theorem]{Facts}
\newtheorem{props}[theorem]{Properties}

\newtheorem{ex}[theorem]{Example}

\newcommand{\notes} {\noindent \textbf{Notes.  }}
\newcommand{\note} {\noindent \textbf{Note.  }}
\newcommand{\defn} {\noindent \textbf{Definition.  }}
\newcommand{\defns} {\noindent \textbf{Definitions.  }}
\newcommand{\x}{{\bf x}}
\newcommand{\z}{{\bf z}}
\newcommand{\B}{{\bf b}}
\newcommand{\V}{{\bf v}}
\newcommand{\T}{\mathbb{T}}
\newcommand{\Z}{\mathbb{Z}}
\newcommand{\Hp}{\mathbb{H}}
\newcommand{\D}{\mathbb{D}}
\newcommand{\R}{\mathbb{R}}
\newcommand{\N}{\mathbb{N}}
\renewcommand{\B}{\mathbb{B}}
\newcommand{\C}{\mathbb{C}}
\newcommand{\dt}{{\mathrm{det }\;}}
 \newcommand{\adj}{{\mathrm{adj}\;}}
 \newcommand{\0}{{\bf O}}
 \newcommand{\av}{\arrowvert}
 \newcommand{\zbar}{\overline{z}}

\newcommand{\ds}{\displaystyle}
\numberwithin{equation}{section}

\renewcommand{\theenumi}{(\roman{enumi})}
\renewcommand{\labelenumi}{\theenumi}

\title{Decomposing diffeomorphisms of the sphere}

\author{Alastair Fletcher, Vladimir Markovic}

\maketitle

\section{Introduction}

\subsection{Background}

A bi-Lipschitz homeomorphism $f:X \rightarrow Y$ between metric spaces is a mapping $f$ such that $f$ and $f^{-1}$ satisfy a uniform Lipschitz condition, that is, there exists $L \geq 1$ such that
\begin{equation*}
\frac{ d_{X}(x,y) }{L} \leq d_{Y}(f(x),f(y)) \leq Ld_{X}(x,y)
\end{equation*}
for all $x,y \in X$. The smallest such constant $L$ is called the {\it isometric distortion} of $f$.
In the metric space setting, a homeomorphism $f:X \rightarrow Y$ is called
quasiconformal if there exists a constant $H \geq 1$ such that
\begin{equation*}
H_{f}(x) : = \limsup _{r \rightarrow 0} \frac
{\sup \{ d_{Y}(f(x),f(y)) : d_{X}(x,y) \leq r \}}
{\inf \{ d_{Y}(f(x),f(y)) : d_{X}(x,y) \geq r \}} \leq H
\end{equation*}
for all $x \in X$. The constant $H$ is called the {\it conformal distortion} of $f$. 
This definition coincides with the perhaps more familiar analytic definition
of quasiconformal mappings in $\R^{n}$.

Let $S^{n}$ be the sphere of dimension $n$ and denote by $QC(S^{n})$ and $LIP(S^{n})$ the orientation preserving quasiconformal and bi-Lipschitz homeomorphisms, respectively, 
of $S^{n}$. An old central problem in this area is the following.

\begin{conjecture}
Let $f$ be in either $QC(S^{n})$ or $LIP(S^{n})$. Then $f$ can be written as a decomposition $f=f_{m} \circ \ldots \circ f_{1}$ where each 
$f_{k}$ has small conformal distortion or isometric distortion respectively.
\end{conjecture}

The conjecture is known for the class $QC(S^{2})$ and is essentially a consequence of solving the Beltrami equation in the plane, see for example
\cite{FM}. The quasisymmetric case $QC(S^{1})$ also follows from the dimension $2$ case.

It is well-known that every $L$-bi-Lipschitz homeomorphism between two intervals
can be factored into bi-Lipschitz mappings with smaller isometric distortion $\alpha$. Such a factorisation can be written explicitly in the following way. 
Let $f:I \rightarrow I'$ be an $L$-bi-Lipschitz mapping. Then $f$ can be written
as $f=f_{2} \circ f_{1}$, where 
\begin{equation*}
f_{1}(x) = \int_{x_{0}}^{x} \av f'(t) \av ^{\lambda}\: dt,
\end{equation*}
$x_{0} \in I$ is fixed, $\lambda = \log _{L} \alpha$, $f_{1}$ is
$\alpha$-bi-Lipschitz and $f_{2} = f \circ f_{1}^{-1}$ is $L/\alpha$-bi-Lipschitz.
It follows that to factorise an $L$-bi-Lipschitz mapping into $\alpha$-bi-Lipschitz mappings requires $N<\log_{\alpha}L + 1$ factors.

In dimension $2$, Freedman and He \cite{FH} studied the logarithmic spiral map
$s_{k}(z) = z e^{ik \log \av z \av}$, which is an $L$-bi-Lipschitz mapping of the plane where $\av k \av = L - 1/L$. They showed that $s_{k}$ requires $N \geq \av k \av (\alpha^{2}-1)^{-1/2}$ factors to be represented as a composition of $\alpha$-bi-Lipschitz mappings. Gutlyanskii and Martio \cite{GM} studied a related class of mappings in dimension $2$, and generalized this to a class of volume preserving bi-Lipschitz automorphisms of the unit ball $\B^{3}$ in $3$ dimensions.
Beyond these particular examples, however, very little is known about factorising
bi-Lipschitz mappings in dimension $2$ and higher, and factorizing quasiconformal maps in dimension $3$ and higher.

A natural question to ask is whether diffeomorphisms of the sphere $S^{n}$ can
be decomposed into diffeomorphisms that are $C^{1}$ close to the identity. The answer in general is negative as the exotic spheres of Milnor \cite{Milnor} provide an obstruction.
In \cite{Milnor}, it is shown that there exist topological $7$-spheres
which are not diffeomorphic to the standard $7$-sphere $S^{7}$.
In particular, one cannot in general find a $C^{1}$ path from the identity on $S^{6}$ to a given $C^{1}$ diffeomorphism. 

There are two facts that might be obstructions to the factorisation theorem. One is the Milnor example. 
The second fact is that not all topological manifolds of dimension at least $5$ admit differentiable structures. On the other hand, a deep result of Sullivan \cite{Sullivan} states that they always admit
a bi-Lipschitz structure. The recent results of Bonk, Heinonen and Wu \cite{Wu} which state that closed bi-Lipschitz manifolds where the transition maps have small enough distortion admit a $C^{1}$ structure,
raises the question of whether a factorisation theorem in this case would contradict Sullivan's theorem.

\subsection{Main results}

Since some $C^{1}$
diffeomorphisms of $S^{n}$ cannot be decomposed into $C^{1}$ diffeomorphisms with derivative close to the identity, that suggests the question of trying to factor them into bi-Lipschitz mappings of small isometric distortion. 

The main result of this paper states that one can find a path
connecting the identity and any $C^{1}$ diffeomorphism of $S^{n}$ which is a composition of bi-Lipschitz paths, a notion that will be made more precise in \S 2.

\begin{theorem}
\label{mainthm}
Let $f:S^{n} \rightarrow S^{n}$ be a $C^{1}$ diffeomorphism. Then there exist  bi-Lipschitz paths $A_{t},p^{1}_{t},p^{2}_{t}:S^{n} \rightarrow S^{n}$ for $t \in [0,1]$ such that $A_{0},p^{1}_{0}$ and $p^{2}_{0}$ are all the identity, and $A_{1} \circ p^{2}_{1} \circ p^{1}_{1} = f$.
\end{theorem}

\begin{remark}
It is not a priori true that a composition of  bi-Lipschitz paths is another bi-Lipschitz path since issues arise at points
of non-differentiability.
\end{remark}

As a corollary to this theorem, we find that
$C^{1}$ diffeomorphisms of the sphere $S^{n}$ can be decomposed into bi-Lipschitz mappings of arbitrarily small
isometric distortion.

\begin{theorem}
\label{maincor}
Let $f:S^{n} \rightarrow S^{n}$ be a $C^{1}$ diffeomorphism. Given $\epsilon >0$, there exists $m \in \N$, depending on $f$, such that
$f$ decomposes as
$f=f_{m} \circ \ldots \circ f_{1}$, where $f_{k}$ is $(1+\epsilon)$-bi-Lipschitz
with respect to the spherical metric $\chi$, and $\chi(f_{k}(x),x) < \epsilon$ for all $x \in S^{n}$ and for $k=1,\ldots,m$.
\end{theorem}




In \S 2, we will state several intermediate lemmas and prove Theorem \ref{mainthm} and Corollary \ref{maincor} assuming these lemmas hold. The proofs of the lemmas are postponed to \S 3. 

\section{Outline of proof}

\subsection{Some notation}

We will first fix some notation. Let $S^{n} = \R^{n} \cup \{ \infty \}$ be the sphere of dimension $n$. 
Denote by $d$ the Euclidean metric on $\R^{n}$ and by $\chi$ the spherical metric on $S^{n}$, so that
\begin{equation*}
d(x,y) = \av x-y \av,
\end{equation*}
for $x,y \in \R^{n}$ and
\begin{equation*}
\chi(x,y) = \frac{\av x-y \av}{\sqrt{1+\av x \av^{2}} \sqrt{1+\av y \av ^{2}}}
\end{equation*}
for $x,y \in S^{n} \setminus \{ \infty \}$. If $y$ is the point at infinity,
\begin{equation*}
\chi(x, \infty) = \frac{1}{\sqrt{1+ \av x \av ^{2}}}.
\end{equation*}
Let $B_{d}(x,r) = \{ y \in \R^{n}: d(x,y)\leq r\} $ and $B_{\chi}(x,r)=\{ y \in S^{n} : \chi (x,y)\leq r\}$ 
be the closed balls centred at $x$ of respectively Euclidean and spherical radius $r$.
We say that a diffeomorphism $f$ is supported on a set $U \subset S^{n}$ if $f$ is the identity on the complement $S^{n}\setminus U$. 

\subsection{Diffeomorphisms supported on balls}

We first need to show that a $C^{1}$ diffeomorphism with a fixed point can be written as a composition of $C^{1}$ diffeomorphisms supported on spherical balls.

\begin{lemma}
\label{lemma1}
Let $f:S^{n} \rightarrow S^{n}$ be a $C^{1}$ diffeomorphism with at least one fixed point. Then there exist $x_{1},x_{2} \in S^{n}$ and $r_{1},r_{2}>0$ such that
$f$ decomposes as
$f=f^{2} \circ f^{1}$ where $f^{1},f^{2}$ are $C^{1}$ diffeomorphisms supported on spherical balls $B_{1}=B_{\chi}(x_{1},r_{1}),B_{2} = B_{\chi}(x_{2},r_{2})$ in $S^{n}$,
and so that neither $B_{1}$ nor $B_{2}$ are $S^{n}$.
\end{lemma}

To prove the lemma, we will need to make use of the following result of Munkres \cite[Lemma 8.1]{Munkres} as formulated in \cite{Wilson}.

\begin{theorem}[\cite{Munkres}]
\label{thm1}
Let $h:\R^{n} \rightarrow \R^{n}$ be an orientation preserving $C^{k}$ diffeomorphism for $1 \leq k \leq \infty$. Then
there exists a $C^{k}$ diffeomorphism $\widetilde{h} :\R^{n} \rightarrow \R^{n}$ which coincides with the identity near $0 \in \R^{n}$
and $h$ near infinity.
\end{theorem}

\begin{proof}[Proof of Lemma \ref{lemma1}]
Suppose that $f:S^{n} \rightarrow S^{n}$ is a $C^{1}$ diffeomorphism
with a fixed point in $S^{n}$. Identifying $S^{n}$ with
$\overline{\R^{n}}$, without loss of generality we can assume $f$ fixes the point at infinity. Then by Theorem \ref{thm1},
there exists a $C^{1}$ diffeomorphism $\widetilde{f}$ and real numbers $r_{1},r_{2}>0$ such that 
$\widetilde{f} \av _{ B_{\chi}(0,r_{1}) }$ is the identity and $\widetilde{f} \av _{ B_{\chi}(\infty, r_{2})}$ is equal to $f$.
We can then write
\begin{equation*}
f =  \left ( f \circ \widetilde{f}^{-1} \right ) \circ \widetilde{f}
\end{equation*}
where $f^{2}:= f \circ \widetilde{f}^{-1}$ is supported on the ball
$S^{n} \setminus B_{\chi}(\infty, r_{2}) $ and $f^{1}:=\widetilde{f}$ is supported on the ball $S^{n} \setminus B_{\chi}(0,r_{1})$. 
\end{proof}

\subsection{Bi-Lipschitz paths}

We shall postpone the proofs of the lemmas in this section until \S 3.
Let us now define the notion of a bi-Lipschitz path.

\begin{definition}
\label{bilippath}
Let $(X,d_{X})$ be a metric space.
A path $h:[0,1] \rightarrow LIP(X)$
is called a \emph{bi-Lipschitz path} if for every $\epsilon >0$, there exists $\delta >0$ such that if $s,t, \in [0,1]$ with $\av s-t \av < \delta$, the following two conditions hold:
\begin{enumerate}
\item for all $x \in X$, $d_{X}(h_{s} \circ h_{t}^{-1}(x),x) <\epsilon$;
\item we have that $h_{s} \circ h_{t}^{-1}$ is $(1+\epsilon)$-bi-Lipschitz with respect to $d_{X}$.
\end{enumerate}
\end{definition}

We need the following lemmas on bi-Lipschitz paths.

\begin{lemma}
\label{lemma10}
Let $h_{t}:[0,1] \rightarrow LIP(\R^{n})$ be a bi-Lipschitz path with respect to $d$.
Then $h_{t}:[0,1] \rightarrow LIP(S^{n})$ is a bi-Lipschitz path with respect to $\chi$.
\end{lemma}

\begin{lemma}
\label{lemma11}
Let $h_{t}:[0,1] \rightarrow LIP(\R^{n})$ be a bi-Lipschitz path with respect to $d$ and let $g:S^{n} \rightarrow S^{n}$ be a M\"{o}bius transformation. Then the path $g \circ h_{t} \circ g^{-1}$
is bi-Lipschitz with respect to $\chi$ on $S^{n}$.
\end{lemma}

\begin{remark}
It can be shown that a bi-Lipschitz path $h_{t}:[0,1] \rightarrow LIP(M)$ on a closed manifold $M$ remains bi-Lipschitz after conjugation by a conformal map $g:M \rightarrow M$. The condition that $g$ is conformal cannot be weakened to $g$ being a diffeomorphism.
\end{remark}

The following lemma is the main step in the proof of Theorem \ref{mainthm}.

\begin{lemma}
\label{thm2}
Let $f:\R^{n} \rightarrow \R^{n}$ be a $C^{1}$ diffeomorphism supported in $B_{d}(0,1/3)$.
Then there exists a path $h_{t}:[0,1] \rightarrow LIP(\R^{n})$ which is bi-Lipschitz with respect to $d$, connecting the identity $h_{0}$ and $h_{1}=f$.
\end{lemma}

\subsection{Proofs of the main results}

Assuming the intermediate results above, the proof of Theorem \ref{mainthm}
proceeds as follows.

\begin{proof}[Proof of Theorem \ref{mainthm}]
Let $f:S^{n} \rightarrow S^{n}$ be a $C^{1}$ diffeomorphism. 
There exists $A \in SO(n)$ such that $A \circ f$ has a fixed point in $S^{n}$. Note that if $n$ is even, then $f$ automatically has a fixed point and we can take $A$ to be the identity.

By Lemma
\ref{lemma1}, we can write $A \circ f=f^{2}\circ f^{1}$ where $f^{i}$ is supported on the spherical ball $B_{i}$ for $i=1,2$.
By standard spherical geometry, see e.g. \cite{V}, for $i=1,2$, there exist M\"{o}bius transformations
$g_{i}$ such that $g_{i}^{-1} \circ f^{i} \circ g_{i}$ is supported on
$B_{d}(0,1/3)$. 

Now, applying Lemma \ref{thm2} to $g_{i}^{-1} \circ f^{i} \circ g_{i}$, we obtain
two bi-Lipschitz paths $h_{t}^{i}$, for $i=1,2$, with respect to $d$ on $\R^{n}$.
Consider the paths
\begin{equation*}
p_{t}^{i} = g_{i} \circ h_{t}^{i} \circ g_{i}^{-1} 
\end{equation*}
for $i=1,2$, where $p_{0}^{i}$ is the identity and $p_{1}^{i} = f^{i}$.

It follows by Lemma \ref{lemma11} that $p_{t}^{i}$ is bi-Lipschitz with respect to $\chi$ on $S^{n}$.
Then $p_{t}^{2} \circ p_{t}^{1}$ is a composition of bi-Lipschitz paths,
with respect to $\chi$,
connecting the identity and $A \circ f$. Since $A^{-1} \in SO(n)$, there is a bi-Lipschitz path $A_{t}$ connecting the identity $A_{0}$ and $A_{1} = A^{-1}$. We conclude
that $A_{t} \circ p_{t}^{2} \circ p_{t}^{1}$ is a composition of three bi-Lipschitz paths, which connects the identity and $f$. This completes the proof.
\end{proof}

\begin{proof}[Proof of Theorem \ref{maincor}]
Let $\epsilon >0$. 
By Theorem \ref{mainthm}, $A_{t},p^{1}_{t}$ and $p^{2}_{t}$ are all bi-Lipschitz paths with respect to $\chi$ on $S^{n}$,  
$A_{0} \circ p_{0}^{2} \circ p_{0}^{1}$ is the identity and
$A_{1} \circ p_{1}^{2} \circ p_{1}^{1}=f$.

Given a bi-Lipschitz path $h_{t}$, we can choose
$0=t_{1}<t_{2}<\ldots < t_{j+1}=1$ such that $g_{k} = h_{k+1} \circ h_{k}^{-1}$ is $(1+\epsilon)$-bi-Lipschitz for $k=1,\ldots, j$ and
$h_{1}= g_{j} \circ  \ldots \circ g_{1}$.
Applying this observation to the bi-Lipschitz paths $A_{t},p^{1}_{t}$ and $p^{2}_{t}$, there exists $j(1),j(2),j(3) \in\N$ such that
\begin{align*}
A_{1} &= A_{1,j(1)} \circ A_{1,j(1)-1} \circ \ldots \circ A_{1,1},\\
p_{1}^{1} &= p_{1,j(2)}^{1} \circ p_{1,j(2)-1}^{1} \circ \ldots \circ p_{1,1}^{1},\\
p_{1}^{2} &= p_{1,j(3)}^{2} \circ p_{1,j(3)-1}^{2} \circ \ldots \circ p_{1,1}^{2},
\end{align*}
and each map in these three decompositions is $(1+\epsilon)$-bi-Lipschitz with respect to $\chi$, and also only moves points in $S^{n}$ by at most spherical distance $\epsilon$. In view of $A_{1} \circ p_{1}^{2} \circ p_{1}^{1}=f$, this proves the theorem with $m=j(1)+j(2)+j(3)$.
\end{proof}

\section{Proofs of the Lemmas}

We will prove Lemma \ref{lemma10} and Lemma \ref{lemma11} first, before proving the main Lemma \ref{thm2}.

\subsection{Proof of Lemma \ref{lemma10}}

Let $h_{t}:\R^{n} \rightarrow \R^{n}$ be a bi-Lipschitz path with respect to $d$. Then each $h_{t}$
extends to a mapping $S^{n} \rightarrow S^{n}$ which fixes the point at infinity.
Let $s,t \in [0,1]$ and consider the mapping $g= h_{s} \circ h_{t}^{-1}$.
Since $h_{t}$ is a bi-Lipschitz path, choose $\delta>0$ small enough so that if $\av s-t \av < \delta$ then
$d(g(x),x) < \epsilon$ for all $x \in \R^{n}$ and $g$ is $(1+\epsilon)$-bi-Lipschitz
with respect to $d$.

Property (i) of Definition \ref{bilippath} is satisfied for $\chi$ since $\chi(g(x),x) \leq d(g(x),x)$, for $x \in \R^{n}$, and $g$ fixes the point at infinity.

We now show that $h_{t}$ satisfies property (ii) of Definition \ref{bilippath}. The fact that $h_{t}$ is a bi-Lipschitz path with respect to $d$ and the formula for the spherical distance give
\begin{align}
\label{l2eq1}
\chi(g(x),g(y)) &= \frac{ \av g(x) - g(y) \av }{\sqrt{1+\av g(x) \av ^{2}}\sqrt{1+ \av g(y) \av ^{2}}}   
\notag \\ &\leq \frac{ (1+\epsilon) \av x- y \av}{\sqrt{1+\av g(x) \av ^{2}}\sqrt{1+ \av g(y) \av ^{2}}} 
\notag \\ &= (1+\epsilon)\chi(x,y) \left ( \frac{ 1+ \av x \av ^{2}}{1 + \av g(x) \av ^{2}} \right )^{1/2}
\left ( \frac{ 1+ \av y \av ^{2}}{1 + \av g(y) \av ^{2}} \right )^{1/2},
\end{align}
for $x,y \in \R^{n}$.
Since $d(g(x),x) < \epsilon$, it follows that
\begin{equation*}
\frac{1+\av x \av ^{2}}{1+ (\av x \av +\epsilon  ) ^{2}} \leq
\frac{1+\av x \av^{2}}{1+\av g(x) \av^{2}} \leq \frac{1+\av x \av ^{2}}{1+ (\av x \av -\epsilon ) ^{2}}.
\end{equation*}
Therefore,
\begin{equation*}
\left ( 1 + \frac{\epsilon(\epsilon + 2 \av x \av)}{1+\av x \av^{2}} \right)^{-1} 
\leq \frac{ 1+ \av x \av ^{2}}{1 + \av g(x) \av ^{2}}
\leq \left ( 1 + \frac{\epsilon(\epsilon - 2 \av x \av)}{1+\av x \av^{2}} \right)^{-1}
\end{equation*}
and so it follows that given $\epsilon>0$, we can choose $\epsilon'$ small enough so that
\begin{equation}
\label{l2eq2}
\frac{1}{1+\epsilon'} \leq  \frac{1+\av x \av^{2}}{1+ \av g(x) \av ^{2}} \leq 1+\epsilon'
\end{equation}
for all $x \in \R^{n}$. By (\ref{l2eq1}) and (\ref{l2eq2}), it follows that 
\begin{equation}
\label{l2eq4}
\chi(g(x),g(y)) \leq (1+\epsilon)(1+\epsilon')\chi(x,y),
\end{equation}
for all $x,y \in \R^{n}$. We can conclude that given $\epsilon >0$,
we can choose $\xi >0$ small enough so that 
\begin{equation}
\label{l2eq3}
\chi(g(x),g(y)) \leq (1+\xi) \chi(x,y)
\end{equation}
for all $x,y \in \R^{n}$. The reverse inequality follows by applying (\ref{l2eq3}) to $g^{-1}$. Therefore condition (ii) of Definition \ref{bilippath} holds for $x,y \in \R^{n}$ with $\delta$, and $\xi$ playing the role of $\epsilon$.

Finally, if $x \in \R^{n}$ and $y= \infty$, then
\begin{equation*}
\chi(g(x),\infty) = \frac{1}{\sqrt{1+\av g(x) \av^{2}}} = \chi(x,\infty) \left (\frac{1+\av x \av^{2}}{1+\av g(x) \av ^{2}}\right) ^{1/2}
\end{equation*}
and we then apply (\ref{l2eq2}) as above. This completes the proof of Lemma \ref{lemma10}.

\subsection{Proof of Lemma \ref{lemma11}}

Recall that $h_{t}$ is a bi-Lipschitz path with respect to $d$ on $\R^{n}$ and that $g:S^{n} \rightarrow S^{n}$ is a M\"{o}bius transformation. We can write
\begin{equation*}
g = C \circ B,
\end{equation*}
where $B:\R^{n} \rightarrow \R^{n}$ is an affine map and $C$
is a spherical isometry. To see this, let $x\in S^n$ be the point
such that $g(\infty)=x$. Then there exists a (non-unique)
spherical isometry $C$ such that $C(\infty)=x$ and then the map $B=C^{-1} \circ g$ is affine.

We first show that $B \circ h_{t} \circ B^{-1}$ is a bi-Lipschitz path with respect to $d$ on $\R^{n}$.
Since $B:\R^{n} \rightarrow \R^{n}$ is an affine map, there is a real number $\alpha >0$ such that
\begin{equation*}
d(B(x),B(y)) = \alpha d(x,y),
\end{equation*}
for all $x,y \in \R^{n}$. Since $h_{t}$ is a bi-Lipschitz path
with respect to $d$, write $f = h_{s} \circ h_{t}^{-1}$, 
with $\av s-t \av < \delta$ small enough so that $d(f(x),x) < \epsilon$ and $f$ is $(1+\epsilon)$-bi-Lipschitz with respect to $d$. Then
\begin{align*}
d(B(f(B^{-1}(x))),x) &= d( B(f(B^{-1}(x))),B(B^{-1}(x)) \\
&\leq \alpha d(f(B^{-1}(x)),B^{-1}(x)) \\
&< \alpha \epsilon,
\end{align*}
for all $x \in \R^{n}$. Therefore $B \circ h_{t} \circ B^{-1}$
satisfies condition (i) of Definition \ref{bilippath} with $\delta$ and $\alpha \epsilon$. Next,
\begin{align*}
d(B(f(B^{-1}(x))),B(f(B^{-1}(y)))) &= \alpha d(f(B^{-1}(x)),f(B^{-1}(y))) \\
&\leq \alpha (1+\epsilon) d(B^{-1}(x),B^{-1}(y)) \\
&= (1+\epsilon) d(x,y)
\end{align*}
and so $B \circ h_{t} \circ B^{-1}$
satisfies condition (ii) of Definition \ref{bilippath} with $\delta$ and $\epsilon$.

By Lemma \ref{lemma10}, $B \circ h_{t} \circ B^{-1}$ is also bi-Lipschitz with respect to $\chi$ on $S^{n}$.
It remains to show that $C \circ B \circ h_{t} \circ B^{-1} \circ C^{-1}= g \circ h_{t} \circ g^{-1}$ is a bi-Lipschitz path with respect to $\chi$ on $S^{n}$.

Since $B \circ h_{t} \circ B^{-1}$ is a bi-Lipschitz path with respect to $\chi$,
write $f = B \circ h_{s} \circ h_{t}^{-1} \circ B^{-1}$, 
with $\av s-t \av < \delta$ small enough so that $\chi(f(x),x) \leq \epsilon$ and $f$ is $(1+\epsilon)$-bi-Lipschitz with respect to $\chi$.
Then
\begin{align*}
\chi (C(f(C^{-1}(x))),x) &= \chi ( C(f(C^{-1}(x))),C(C^{-1}(x)))\\
&= \chi (f(C^{-1}(x)),C^{-1}(x)) \\
&< \epsilon,
\end{align*}
for all $x \in S^{n}$. Therefore $C \circ B \circ h_{t} \circ B^{-1} \circ C^{-1}$
satisfies condition (i) of Definition \ref{bilippath} with $\delta$ and $\epsilon$. Next,
\begin{align*}
\chi ( C(f(C^{-1}(x))), C(f(C^{-1}(y))) ) &= \chi (f(C^{-1}(x)), f(C^{-1}(y)) )\\
&\leq (1+\epsilon) \chi ( C^{-1}(x),C^{-1}(y))\\
&= (1+\epsilon)\chi (x,y),
\end{align*}
and so $C \circ B \circ h_{t} \circ B^{-1} \circ C^{-1}$
satisfies condition (ii) of Definition \ref{bilippath} with $\delta$ and $\epsilon$. This completes the proof.

\subsection{Proof of Lemma \ref{thm2}}

We first set some notation. If $g:\R^{n} \rightarrow \R^{n}$ is differentiable at $x \in \R^{n}$,
write $D_{x}g$ for the derivative of $g$ at $x$ and let
\begin{equation*}
\av \av D_{x}g \av \av = \max_{y \in \R^{n} \setminus \{0 \} } \frac{ \av (D_{x}g)(y) \av}{\av y \av} 
\end{equation*}
be the operator norm of the linear map
$D_{x}g$. Note that we are regarding the derivative here as a mapping from
$\R^{n}$ to $\R^{n}$ given by the matrix of partial derivatives $\partial g_{i}/ \partial x_{j}$, and not as a mapping between tangent spaces.

Recall that $f:\R^{n} \rightarrow \R^{n}$ is a $C^{1}$ diffeomorphism supported on the 
ball $B_{0}:=B_{d}(0,1/3)$. Write $A_{t}:\R^{n} \rightarrow \R^{n}$ for the translation $A_{t}(x_{1},x_{2},\ldots,x_{n}) = (x_{1}+t,x_{2},\ldots,x_{n})$
and define $B_{t} = A_{t}(B_{0})$. Write $e_{1}=(1,0,\ldots,0)$.

Define $g:\R^{n} \rightarrow \R^{n}$ by
\begin{equation*}
g(x) = \left\{ \begin{array}{cl}
 (A_{m} \circ f \circ A_{m}^{-1})(x) &\mbox{ if $x \in B_{m}$, \:\:\: $m \in \N$,} \\
  x &\mbox{ otherwise.}
       \end{array} \right.
\end{equation*}
Then $g$ is a propagated version of $f$, supported in $\cup _{m=1}^{\infty} B_{m}$. We can extend $g$ to a mapping on $S^{n}$ by defining $g$ to fix the point at infinity.

\begin{lemma}
\label{lemma3}
The map $g$ is $C^{1}$ on $\R^{n}$ and, further, satisfies the following properties:
\begin{enumerate}
\item $g$ is uniformly continuous on $\R^{n}$, that is, for all $\epsilon >0$,
there exists $\delta >0$ such that for all $x,y \in \R^{n}$ satisfying $\av x-y \av < \delta$,
we have $\av g(x)-g(y) \av <\epsilon$;
\item there exists $T>0$ such that 
\begin{equation}
\label{l3eq1}
\av \av D_{x}g\av \av \leq T 
\end{equation}
for all $x \in \R^{n}$;
\item there exists a function $\eta :[0,\infty]\rightarrow [0,\infty]$ for which $\eta(0)=0$, $\eta$ is continuous at $0$ and 
\begin{equation}
\label{l3eq2}
\av \av D_{x}g - D_{y}g \av \av \leq \eta (\av x-y\av)
\end{equation}
for all $x,y \in \R^{n}$. The function $\eta$ is the modulus of continuity of $Dg$.
\end{enumerate}
Further, we may assume that $g^{-1}$ also satisfies these three conditions, by changing the constants and modulus of continuity if necessary.
\end{lemma}

\begin{proof}
First note that $f$ is $C^{1}$ by hypothesis, and satisfies the three claims of the lemma because it is supported in a compact subset of $\R^{n}$.
Since $g$ is a propagated version of $f$, it satisfies the three claims of the lemma
with the same constants as $f$.
The last claim follows since $f^{-1}$ is also $C^{1}$, and $g^{-1}$ is a propagated version of $f^{-1}$.
\end{proof} 

\begin{definition}
For $t \in [0,1]$, let
\begin{equation*}
h_{t} = g^{-1} \circ A_{t}^{-1} \circ g \circ A_{t}.
\end{equation*}
\end{definition}
By Lemma \ref{lemma3} and \cite[Lemma 1.54]{V}, which says that Euclidean translations in $\R^{n}$ are bi-Lipschitz with respect to $\chi$, $h_{t}$ is bi-Lipschitz with respect to both $d$ and $\chi$.
The following lemma is elementary.

\begin{lemma}
\label{lemma9}
We have that $h_{0}$ is equal to the identity and $h_{1} =f$. 
\end{lemma}

Observe that $h_{t}$ is a path that connects the identity and $f$ through bi-Lipschitz mappings, for $0 \leq t \leq 1$. We now want to show that this is a bi-Lipschitz path.

\begin{lemma}
\label{lemma6}
Given $\epsilon >0$, there exists $\delta>0$ such that if $s,t \in [0,1]$ satisfy $\av s-t \av <\delta$, then
\begin{equation*}
d( h_{s} \circ h_{t}^{-1}(x), x)  \leq \epsilon,
\end{equation*}
for all $x \in \R^{n}$.
\end{lemma}

\begin{proof}
Writing $h_{s}\circ h_{t}^{-1}$ out in full gives
\begin{equation}
\label{l6eq1}
h_{s} \circ h_{t}^{-1}  = g^{-1} \circ A_{s}^{-1} \circ g \circ A_{s} \circ A_{t}^{-1} \circ g^{-1} \circ A_{t} \circ g.
\end{equation}
Considering first the middle four functions in this expression, write 
\begin{equation}
\label{l6eq2}
P_{s,t}(x) = g \circ A_{s} \circ A_{t}^{-1} \circ g^{-1}(x).
\end{equation}
Then the fact that 
\begin{equation*}
d(g(x),g(y)) \leq \sup_{x} \av \av D_{x}g \av \av \cdot d(x,y), 
\end{equation*}
and (\ref{l3eq1}) gives
\begin{align*}
d( P_{s,t}(x), x ) &=
d( g(g^{-1}(x) + (s-t)e_{1}) , g(g^{-1}(x)) ) \\
&\leq T d( g^{-1}(x) + (s-t)e_{1} , g^{-1}(x) ) \\
&= T \av s-t \av,
\end{align*}
for all $x \in \R^{n}$.
Next, by using the the fact that translations are isometries of $\R^{n}$, the triangle inequality and the previous inequality applied to $x+te_{1}$, we obtain 
\begin{align}
\label{l3eq3}
\notag d( A_{s}^{-1} \circ P_{s,t} \circ A_{t}(x) , x) &= d( P_{s,t}(x+te_{1}) - se_{1},x)  \\ \notag
&= d( P_{s,t}(x+te_{1}) , (x+te_{1}) + (s-t)e_{1} )\\ \notag
&\leq d( P_{s,t}(x+te_{1}) , (x+te_{1})) + d(x+te_{1},x+te_{1} + (s-t)e_{1}) \\
&\leq (T+1)\av s-t \av,
\end{align}
for all $x \in \R^{n}$.
Finally, we use (\ref{l3eq1}) with $g^{-1}$ and (\ref{l3eq3}) applied to $g(x)$ to obtain
\begin{align*}
d( h_{s} \circ h_{t}^{-1}(x),x ) &= d(g^{-1} \circ A_{s}^{-1} \circ P_{s,t} \circ A_{t} \circ g(x) , g^{-1}(g(x))) \\
&\leq T d (A_{s}^{-1} \circ P_{s,t} \circ A_{t} \circ g(x) , g(x) ) \\
&\leq T(T+1)\av s-t \av,
\end{align*}
for all $x \in \R^{n}$. We can therefore take $\delta = \epsilon / T(T+1)$.
\end{proof}

\begin{lemma}
\label{lemma7}
Given $\epsilon >0$, there exists $\delta>0$ such that if $s,t \in [0,1]$ satisfy $\av s-t \av <\delta$, 
then 
\begin{equation*}
\av \av D_{x}(h_{s} \circ h_{t}^{-1}) - I \av \av < \epsilon
\end{equation*}
for all $x \in \R^{n}$, where $I$ is the identity mapping.
\end{lemma}

\begin{proof}
Recalling the strategy of the proof of the previous lemma, we will consider the middle six terms of (\ref{l6eq1}) and work outwards. Recall the definition of $P_{s,t}$ from (\ref{l6eq2}) and write $Q_{s,t}=A_{s}^{-1} \circ P_{s,t} \circ A_{t}$.
Observe that
\begin{equation*}
D_{x}Q_{s,t} = D_{A_{t}(x)}P_{s,t}
\end{equation*}
and
\begin{equation*}
D_{x}P_{s,t} = D_{A_{s} \circ A_{t}^{-1} \circ g^{-1}(x)} g \circ D_{x}g^{-1}
\end{equation*}
since the derivative of $A_{t}$ is the identity.
By this observation, the chain rule gives
\begin{equation}
\label{l3eq4}
\av \av D_{x}(Q_{s,t}) - I \av \av 
= \av \av (D_{A_{s}\circ A_{t}^{-1}\circ g^{-1} \circ A_{t}(x)}g) \circ (D_{A_{t}(x)}g^{-1}) - I \av \av.
\end{equation}
We can write the right hand side of (\ref{l3eq4}) as
\begin{equation*}
\av \av \left [ (D_{A_{s}\circ A_{t}^{-1}\circ g^{-1} \circ A_{t}(x)}g) - \left( (D_{A_{t}(x)}g^{-1}) \right )^{-1} \right ] \circ (D_{A_{t}(x)}g^{-1}) \av \av.
\end{equation*}
Using this, and applying the formula for the derivative of an inverse $(D_{A_{t}(x)}g^{-1})^{-1} = D_{g^{-1}(A_{t}(x))}g$ and 
(\ref{l3eq1}) applied to $g^{-1}$, yields from (\ref{l3eq4}) that
\begin{equation}
\label{l3eq5}
\av \av D_{x}(Q_{s,t}) - I \av \av \leq 
T \av \av (D_{A_{s}\circ A_{t}^{-1}\circ g^{-1} \circ A_{t}(x)}g) - (D_{g^{-1}\circ A_{t}(x)}g) \av \av .
\end{equation}
We then apply (\ref{l3eq2}) to the right hand side of 
(\ref{l3eq5}) to give
\begin{align}
\label{l3eq6}
\av \av D_{x}(Q_{s,t}) - I \av \av 
&\leq T \eta ( \av A_{s} \circ A_{t}^{-1} \circ g^{-1} \circ A_{t}(x) -g^{-1} \circ A_{t}(x) \av )\\
\notag &= T \eta(\av s-t\av),
\end{align}
for all $x \in \R^{n}$.

Now, consider the derivative of
$h_{s} \circ h_{t}^{-1} = g^{-1} \circ Q_{s,t} \circ g$.
By the chain rule, we have
\begin{equation}
\label{l3eq7}
\av \av D_{x}(g^{-1} \circ Q_{s,t} \circ g) - I \av \av
= \av \av (D_{Q_{s,t}(g(x))}g^{-1}) \circ (D_{g(x)}Q_{s,t}) \circ (D_{x}g) - I \av \av.
\end{equation}
We can write the right hand side of (\ref{l3eq7}) as
\begin{equation*}
\av \av (D_{Q_{s,t}(g(x))}g^{-1}) \circ \left [ D_{g(x)}Q_{s,t} - I \right ] \circ (D_{x}g) 
 + (D_{Q_{s,t}(g(x))}g^{-1}) \circ (D_{x}g) - I \av \av.
\end{equation*}
Applying the triangle inequality and (\ref{l3eq1}) for $g$ and $g^{-1}$ to this expression yields
\begin{equation}
\label{l3eq8}
\av \av D_{x}(g^{-1} \circ Q_{s,t} \circ g) - I \av \av
\leq T^{2} \av \av  D_{g(x)}Q_{s,t} - I \av \av + \av \av (D_{Q_{s,t}(g(x))}g^{-1}) \circ (D_{x}g) - I \av \av
\end{equation}
We next apply (\ref{l3eq6}) to the first term on the right hand side of (\ref{l3eq8}), and re-write the second term to give
\begin{equation}
\label{l3eq9}
\av \av D_{x}(g^{-1} \circ Q_{s,t} \circ g) - I \av \av
\leq T^{3} \eta(\av s-t \av) + \av \av \left [ D_{Q_{s,t}(g(x))}g^{-1} - (D_{x}g)^{-1} \right ] \circ (D_{x}g) \av \av
\end{equation}
We use the formula $(D_{x}g)^{-1} = D_{g(x)}g^{-1}$
and (\ref{l3eq1}) applied to $g$ on the second term 
on the right hand side of (\ref{l3eq9}) to yield
\begin{equation*}
\av \av D_{x}(g^{-1} \circ Q_{s,t} \circ g) - I \av \av 
\leq T^{3} \eta(\av s-t \av) + T \av \av D_{Q_{s,t}(g(x))}g^{-1} - D_{g(x)}g^{-1} \av \av 
\end{equation*}
Finally, (\ref{l3eq2}) and (\ref{l3eq3}) give
\begin{align*}
\av \av D_{x}(g^{-1} \circ Q_{s,t} \circ g) - I \av \av & \leq T^{3} \eta(\av s-t \av) + T\eta( \av Q_{s,t}(g(x)) - g(x) \av )\\
&\leq T^{3} \eta(\av s-t \av )+ T \eta ( (T+1) \av s-t \av).
\end{align*}
Since $\lim _{x \rightarrow 0} \eta(x)=0$, the lemma follows. 
\end{proof}

Lemmas \ref{lemma9}, \ref{lemma6} and \ref{lemma7} together show that $h_{t}$
is a bi-Lipschitz path with respect to $d$ connecting the identity
and $f$.
This completes the proof.

\end{document}